\documentclass{amsart}
\usepackage{setspace}
\usepackage{a4}
\usepackage{amsthm}
\usepackage{latexsym}
\usepackage{amsfonts}
\usepackage{graphicx}
\usepackage{textcomp}
\usepackage{cite}
\usepackage{enumerate}
\usepackage{amssymb}
\usepackage{hyperref}
\usepackage{amsmath}
\usepackage{tikz}
\usepackage[mathscr]{euscript}
\usepackage{mathtools}
\newtheorem{theorem}{Theorem}[section]

\newtheorem{conjecture}[theorem]{Conjecture}

\newtheorem{definition}[theorem]{Definition}

\newtheorem{problem}[theorem]{Problem}

\newtheorem{remark}[theorem]{Remark}
\newtheorem{question}[theorem]{Question}
\title{This is the title}
\usepackage{fancyhdr}

\begin{document}
\hrule\hrule\hrule\hrule\hrule
\vspace{0.3cm}	
\begin{center}
{\bf{ABSOLUTELY SUMMING MORPHISMS BETWEEN HILBERT C*-MODULES AND MODULAR PIETSCH FACTORIZATION PROBLEM}}\\
\vspace{0.3cm}
\hrule\hrule\hrule\hrule\hrule
\vspace{0.3cm}
\textbf{K. MAHESH KRISHNA}\\
Post Doctoral Fellow \\
Statistics and Mathematics Unit\\
Indian Statistical Institute, Bangalore Centre\\
Karnataka 560 059, India\\
Email: kmaheshak@gmail.com\\

Date: \today
\end{center}

\hrule\hrule
\vspace{0.5cm}
\textbf{Abstract}: Motivated from the theory of Hilbert-Schmidt  morphisms between Hilbert C*-modules over commutative C*-algebras by Stern and van Suijlekom \textit{[J. Funct. Anal., 2021]}, we introduce the notion of p-absolutely summing morphisms between Hilbert C*-modules over commutative C*-algebras. We show that  an adjointable morphism between Hilbert C*-modules over monotone closed commutative C*-algebra is 2-absolutely summing if and only if   it is Hilbert-Schmidt. We formulate version of Pietsch factorization problem  for p-absolutely summing morphisms and solve partially.



\hrule

\hrule
\section{Introduction}
In his \textit{Resume}, A. Grothendieck studied 1-absolutely and 2-absolutely summing operators between Banach spaces \cite{GROTHENDIECK} (also see \cite{DIESTELFOURIESWART}). In 1967, for each $1\leq p<\infty$, A. Pietsch introduced the notion of p-absolutely summing operators which became an area around the end of 20 century  \cite{TOMCZAK,  DEFANTFLORET, WOJTASZCZYK, LINDENSTRAUSSPELCZYNKSKI, PIETSCH, DIESTELJARCHOWTONGE, PELCZYNSKI, JOHNSONSCHECHTMAN, PIETSCH2, HANDBOOK1, HANDBOOK2, PIETSCHHISTORY, HYTONENVANNEERVENVERAARWEIS, JAMESON, GARLING, LIQUEFFELEC2, LIQUEFFELEC1, LINDENSTRAUSSTZAFRIRI}. In 1979, Tomczak-Jaegermann studied p-summing operators by fixing by fixing  number of points  \cite{TOMCZAK2}.  In 1970, Kwapien defined the notion of 0-summing operators \cite{KWAPIEN}. In 2003, Farmer and Johnson introduced the notion of Lipschitz p-summing operators between metric spaces \cite{FARMERJOHNSON} (also see  \cite{BOTELHOPELLEGRINORUEDA, PELLEGRINOSANTOS, CHENZHENG, BOTELHOMAIAPELLEGRINOSANTOS}).
   \begin{definition}\cite{ALBIACKALTON, PIETSCH}\label{PSUMMING}
  	Let $\mathcal{X}$ and $\mathcal{Y}$ be Banach spaces, $\mathcal{X}^*$ be the dual of $\mathcal{X}$ and $1\leq p<\infty$. A bounded linear operator $T:\mathcal{X} \to \mathcal{Y}$ is said to be \textbf{p-absolutely summing} if there is a real constant $C>0$ satisfying following: for every   $n\in \mathbb{N}$ and for all $x_1, \dots, x_n \in \mathcal{X}$, 
  	\begin{align}\label{ABSUMEQ}
  	\left(\sum_{j=1}^{n}\|Tx_j\|^p\right)^\frac{1}{p}\leq C \sup_{f\in \mathcal{X}^*, \|f\|\leq 1}	\left(\sum_{j=1}^{n}|f(x_j)|^p\right)^\frac{1}{p}.
  	\end{align}
 In this case, the \textbf{p-absolutely summing norm} of $T$ is defined as 
\begin{align*}
	\pi_p(T)\coloneqq \inf \{C: C\text{ satisfies Inequality } (\ref{ABSUMEQ})\}.
\end{align*}
The set of all p-absolutely summing operators from $\mathcal{X}$ to  $\mathcal{Y}$ is denoted by $\Pi_p(\mathcal{X}, \mathcal{Y})$. 
 \end{definition}
Following are most important results in the theory of p-absolutely summing operators. 
\begin{theorem}\cite{PIETSCH2, PIETSCH}
Let $1\leq p <\infty $ and    $\mathcal{X}$, $\mathcal{Y}$ be Banach spaces. Then   $(\Pi_p(\mathcal{X}, \mathcal{Y}), \pi_p(\cdot))$	is an operator ideal. 
\end{theorem}
\begin{theorem}\cite{ALBIACKALTON, PIETSCH} \label{2IFFHS}
Let $\mathcal{H}$ and $\mathcal{K}$ be Hilbert spaces. Then a bounded linear operator $T:\mathcal{H}\to \mathcal{K}$ is  2-absolutely summing if and only if it is Hilbert-Schmidt. Moreover,  $\|T\|_{\text{HS}}= \pi_2(T)$. 
\end{theorem}
\begin{theorem}\cite{ALBIACKALTON, PIETSCH}\label{PETSCHFACTORIZATIONTHEOREM}
(\textbf{Pietsch Factorization Theorem})	Let $\mathcal{X}$ and $\mathcal{Y}$ be Banach spaces. A bounded linear operator $T:\mathcal{X} \to \mathcal{Y}$ is  p-absolutely summing if and only if there is a real constant $C>0$ and    a regular Borel probability measure on $B_{\mathcal{X}^*}\coloneqq \{f:f \in \mathcal{X}^*, \|f\|\leq 1\}$ in weak*-topology  such that 
\begin{align}\label{FACINE}
\|Tx\|\leq C \left(\int\limits_{B_{\mathcal{X}^*}}|f(x)|^p\,d\mu_{B_{\mathcal{X}^*}}(f)\right)^\frac{1}{p}, \quad \forall x \in \mathcal{X}.
\end{align}
Moreover,  $	\pi_p(T)=\inf \{ C: C \text{ satisfies Inequality }(\ref{FACINE})\}$.
\end{theorem}
In this paper, we define the notion of p-absolutely summing morphisms between Hilbert C*-modules over commutative C*-algebras (Definition \ref{MOD2}). We derive in Theorem \ref{MOD2HS} that an adjointable morphism between Hilbert C*-module over a monote closed C*-algebra  is 2-summing if and only if modular Hilbert-Schmidt. We then formulate version of Pietsch factorization problem for p-absolutely summing morphisms and solve partially.

  \section{p-absolutely summing morphisms}
  We define modular version of Definition  \ref{PSUMMING}  as follows. For the theory of Hilbert C*-modules we refer \cite{PASCHKE, RIEFFEL, KAPLANSKY}.
  \begin{definition}\label{MOD2}
Let $1\leq p <\infty$.	Let $\mathcal{M}$ and $\mathcal{N}$ be Hilbert C*-modules over a commutative C*-algebra $\mathcal{A}$. An adjointable morphism   $T:\mathcal{M} \to \mathcal{N}$ is said to be \textbf{modular p-absolutely summing} if there is a real constant $C>0$ satisfying following: for every   $n\in \mathbb{N}$ and for all $x_1, \dots, x_n \in \mathcal{M}$, 
\begin{align}\label{MODABSUMEQ}
	\left\|\sum_{j=1}^{n}\langle Tx_j, Tx_j \rangle^\frac{p}{2}\right\|^\frac{1}{p}\leq C \sup_{x\in \mathcal{M}, \|x\|\leq 1}	\left\|\sum_{j=1}^{n}(\langle x, x_j \rangle \langle x_j, x  \rangle)^\frac{p}{2} \right\|^\frac{1}{p}.
\end{align}
In this case, the p-absolutely summing norm of $T$ is defined as 
\begin{align*}
	\pi_p(T)\coloneqq \inf \{C: C\text{ satisfies Inequality } (\ref{MODABSUMEQ})\}.
\end{align*}
The set of all p-absolutely summing morphisms from $\mathcal{M}$ to  $\mathcal{N}$ is denoted by $\Pi_p(\mathcal{M}, \mathcal{N})$. 	
\end{definition}  
In 2021, Stern and van Suijlekom introduced the notion of modular Schatten class morphisms \cite{STERNVANSUIJEKOM}. 
\begin{definition}\cite{STERNVANSUIJEKOM}
Let $1\leq p <\infty$. Let $\mathcal{A}$ be a C*-algebra and $\widehat{\mathcal{A}}$ be its Gelfand spectrum. 	Let $\mathcal{M}$ and $\mathcal{N}$ be Hilbert C*-modules over   $\mathcal{A}$. Let $T:\mathcal{M} \to \mathcal{N}$ be an adjointable morphism. We say that $T$ is in 	the \textbf{modular p-Schatten class} if the function
\begin{align*}
	\operatorname{Tr}|T|^p: \widehat{\mathcal{A}}\ni \chi \mapsto \operatorname{Tr}|\chi_*T|^p \in \mathbb{R}\cup \{\infty\}
\end{align*}
lies in $\mathcal{A}$. The modular p-Schatten norm of $T$ is defined as 
\begin{align*}
	\|T\|_p\coloneqq \|\operatorname{Tr}|T|^p\|_\mathcal{A}^\frac{1}{p}.
\end{align*}
Modular 2-Schatten (resp. 1-Schatten) class morphism  is called as \textbf{modular Hilbert-Schmidt} (resp. \textbf{modular trace class}). We denote $	\|T\|_2$ by $\|T\|_{\text{HS}}$.
\end{definition}
 Using the theory of modular frames for Hilbert C*-modules (see \cite{FRANKLARSON}) Stern and van Suijlekom were able to derive following result.
\begin{theorem}\cite{STERNVANSUIJEKOM}\label{SVSTHEOREM}
	Let $\mathcal{M}$ and $\mathcal{N}$ be Hilbert C*-modules over   $\mathcal{A}$. Let $T:\mathcal{M} \to \mathcal{N}$ be an adjointable morphism. 	Then $T$ is modular Hilbert-Schmidt if and only if for every modular Parseval frame  $\{\tau_n\}_n$ for $\mathcal{M}$, the series  $\sum_{n=1}^{\infty}\langle |T|^p\tau_n,  \tau_n \rangle$ converges in norm in $\mathcal{A}$ and 
	\begin{align*}
	\operatorname{Tr}|T|^p=\sum_{n=1}^{\infty}\langle |T|^p\tau_n,  \tau_n \rangle.
	\end{align*}
\end{theorem}
We now derive modular version of Theorem  \ref{2IFFHS} with the following notion. 
\begin{definition}\cite{TAKESAKI}
A C*-algebra $\mathcal{A}$ is said to be \textbf{monotone closed} if every bounded increasing net in $\mathcal{A}$ has the least upper bound in 	$\mathcal{A}$.
\end{definition}
\begin{theorem}\label{MOD2HS}
	Let $\mathcal{M}$ and $\mathcal{N}$ be Hilbert C*-modules over a commutative  C*-algebra $\mathcal{A}$. Assume  $\mathcal{A}$ is monotone closed. Let $T:\mathcal{M} \to \mathcal{N}$ be an adjointable morphism. Then  $T \in \Pi_2(\mathcal{M}, \mathcal{N})$ if and only if  $T$ is modular Hilbert-Schmidt. Moreover,  $\|T\|_{\text{HS}}= \pi_2(T)$. 
 \end{theorem} 
  \begin{proof}
  	$(\Rightarrow)$ Let $T \in \Pi_2(\mathcal{M}, \mathcal{N})$. Let $\{\tau_n\}_{n=1}^\infty$ be a modular Parseval frame for $\mathcal{M}$. Then 
  	\begin{align}\label{FRAMEEQU}
  		\langle x, x \rangle =\sum_{n=1}^{\infty}\langle x, \tau_n \rangle \langle \tau_n, x \rangle, \quad \forall x \in  \mathcal{M},
  	\end{align}
  where the series converges in the norm of $\mathcal{A}$. To show $T$ is modular Hilbert-Schmidt, using Theorem \ref{SVSTHEOREM}, it suffices to show that the series $\sum_{n=1}^{\infty}\langle T\tau_n,  T\tau_n \rangle$ converges in norm in $\mathcal{A}$. Note that the series $\sum_{n=1}^{\infty}\langle T\tau_n,  T\tau_n \rangle$ is monotonically increasing. Since the C*-algebra is monotone closed,  we are done if we show the sequence  $\{\sum_{j=1}^{n}\langle T\tau_j,  T\tau_j \rangle\}_{n=1}^\infty$ is bounded. Let $n \in \mathbb{N}$. Since $T$ is 2-summing, using Equation (\ref{FRAMEEQU}) we have 
  \begin{align*}
  	\left\|\sum_{j=1}^n\langle T\tau_j,  T\tau_j \rangle \right\|&\leq 	\pi_2(T)^2\sup_{x\in \mathcal{M}, \|x\|\leq 1}	\left\|\sum_{j=1}^{n}\langle x, \tau_j \rangle \langle \tau_j, x  \rangle \right\|\leq 	\pi_2(T)^2\sup_{x\in \mathcal{M}, \|x\|\leq 1}	\left\|\sum_{j=1}^{\infty}\langle x, \tau_j \rangle \langle \tau_j, x  \rangle \right\|\\
  	&=\pi_2(T)^2\sup_{x\in \mathcal{M}, \|x\|\leq 1}\|x\|^2=\pi_2(T)^2.
  \end{align*}
$(\Leftarrow)$  Let $n\in \mathbb{N}$ and  $x_1, \dots, x_n \in \mathcal{M}$. Let $\{\omega_n\}_{n=1}^\infty$ be an orthonormal basis for $\mathcal{M}$. Define 
\begin{align*}
	S:\mathcal{M}\ni x \mapsto \sum_{j=1}^n\langle x, \omega_j \rangle x_j\in \mathcal{M}.
\end{align*}
Then 
\begin{align*}
	\|S\|^2&=\|S^*\|^2=\sup_{x\in \mathcal{M}, \|x\|\leq 1}\|S^*x\|^2=\sup_{x\in \mathcal{M}, \|x\|\leq 1}\left\| \sum_{n=1}^\infty \langle S^*x, \omega_n \rangle \omega_n\right\|^2=\sup_{x\in \mathcal{M}, \|x\|\leq 1}\left\| \sum_{n=1}^\infty \langle x, S\omega_n \rangle \omega_n\right\|^2\\
	&=\sup_{x\in \mathcal{M}, \|x\|\leq 1}\left\| \sum_{j=1}^n \langle x, x_j \rangle \omega_j\right\|^2=\sup_{x\in \mathcal{M}, \|x\|\leq 1} \left\| \sum_{j=1}^n\langle x, x_j\rangle \langle x_j, x\rangle \right\|.
\end{align*}
Hence 
\begin{align*}
		\left\|\sum_{j=1}^n\langle Tx_j,  Tx_j \rangle \right\|&=\left\|\sum_{j=1}^n\langle TS\omega_j,  TS\omega_j \rangle \right\|=\|TS\|^2_{\operatorname{HS}}\leq \|T\|^2_{\operatorname{HS}}\|S\|
		=\|T\|^2_{\operatorname{HS}}\sup_{x\in \mathcal{M}, \|x\|\leq 1} \left\| \sum_{j=1}^n\langle x, x_j\rangle \langle x_j, x\rangle \right\|.
\end{align*}
  \end{proof}  
Note that we have not used monotonic closedness of C*-algebra in ``if'' part. In view of Theorem \ref{PETSCHFACTORIZATIONTHEOREM}, we formulate following problem.
\begin{question}\label{QUESTIONPIETSCH}
	\textbf{Whether there exists a modular Pietsch factorization theorem?}
\end{question}
We solve Question (\ref{QUESTIONPIETSCH}) partially in the following theorem. Integrals in the following theorem is in the Kasparov sense \cite{KASPAROV}.

\begin{theorem}
	Let $\mathcal{M}$ and $\mathcal{N}$ be Hilbert C*-modules over a commutative  C*-algebra $\mathcal{A}$. Let $T:\mathcal{M} \to \mathcal{N}$ be an adjointable morphism.  Assume that there exists a Lie group $G\subseteq	B_{\mathcal{M}}\coloneqq \{x:x \in \mathcal{M}, \|x\|\leq 1\}$ satisfying following. 
	\begin{enumerate}[\upshape(i)]
		\item $\mu_G(G)=1$.
		\item For each $x \in \mathcal{M}$, the map $G \ni y \mapsto \langle x. y \rangle \langle y,x \rangle$ is continuous. 
		\item There exists a real $C>0$ such that 
		\begin{align*}
		\langle Tx, Tx \rangle^\frac{p}{2}\leq C^p\int_G (\langle x. y \rangle \langle y,x \rangle)^\frac{p}{2}\,d\mu_G(y), \quad \forall x \in \mathcal{M}.
		\end{align*}
	\end{enumerate}
Then $T$ modular p-absolutely summing  and 	$\pi_p(T)=C$.
\end{theorem}
\begin{proof}
Let $n\in \mathbb{N}$ and  $x_1, \dots, x_n \in \mathcal{M}$. Then 
\begin{align*}
&\left\|\sum_{j=1}^{n}\langle Tx_j, Tx_j \rangle^\frac{p}{2}\right\|\leq C^p\left\|\sum_{j=1}^{n}\int_G (\langle x_j, y \rangle \langle y,x_j \rangle)^\frac{p}{2}\,d\mu_G(y)\right\|= C^p\left\|\int_G\sum_{j=1}^{n} (\langle x_j, y \rangle \langle y,x_j \rangle)^\frac{p}{2}\,d\mu_G(y)\right\| \\
&\leq C^p\int_G\left\|\sum_{j=1}^{n} (\langle x_j, y \rangle \langle y,x_j \rangle)^\frac{p}{2}\right\|\,d\mu_G(y)\leq C^p\int_G\sup_{y\in \mathcal{M}, \|y\|\leq 1} \left\|\sum_{j=1}^{n} (\langle x_j, y \rangle \langle y,x_j \rangle)^\frac{p}{2}\right\|\,d\mu_G(y)\\
&=C^p\sup_{y\in \mathcal{M}, \|y\|\leq 1} \left\|\sum_{j=1}^{n} (\langle x_j, y \rangle \langle y,x_j \rangle)^\frac{p}{2}\right\|\mu_G(G)=C^p\sup_{y\in \mathcal{M}, \|y\|\leq 1} \left\|\sum_{j=1}^{n} (\langle x_j, y \rangle \langle y,x_j \rangle)^\frac{p}{2}\right\|.
\end{align*}	
\end{proof}

  \section{Appendix}
  In this appendix we formulate some problems for Banach modules over C*-algebras based on the results in Banach spaces which influenced a lot in the modern development of Functional Analysis. Our first kind of problems come from the Dvoretzky theorem \cite{MILMANSCHECHTMAN, MILMAM, MILMAN1992, FIGIEL2, SZANKOWSKI, MATOUSEK, MATSAKPLICHKO, PISIERBOOK, GORDONGUEDONMEYER, FIGIEL, FIGIEL3, FIGIELLINDENSTRAUSSMILMAN, SCHECHTMAN, GORDON2, PISIERBOOK2}.
  Let $\mathcal{X}$ and $\mathcal{Y}$ be finite dimensional Banach spaces such that $\dim (\mathcal{X})=\dim(\mathcal{Y})$. Remenber that the Banach-Mazur distance between $\mathcal{X}$ and $\mathcal{Y}$ is defined as 
  \begin{align*}
  	d_{BM}(\mathcal{X}, \mathcal{Y})\coloneqq \inf\{\|T\|\|T^{-1}\|: T: \mathcal{X} \to \mathcal{Y} \text{ is invertible linear operator}\}.
  \end{align*}
  For $n \in \mathbb{N}$, let $(\mathbb{R}^n, \langle \cdot, \cdot \rangle)$ be the standard Euclidean Hilbert space.
  \begin{theorem}\cite{ALBIACKALTON, JOHN} \textbf{(John Theorem)
  		If  $\mathcal{X}$ is any $n$-dimensional real Banach space, then 
  		\begin{align*}
  			d_{BM}(\mathcal{Y}, (\mathbb{R}^n, \langle \cdot, \cdot \rangle))\leq \sqrt{n}.	
  	\end{align*}}
  \end{theorem}
  
  \begin{theorem}\cite{DVORETZKY, ALBIACKALTON} \textbf{(Dvoretzky Theorem)
  		There is a universal constant $C>0$ satisfying the following property: If $\mathcal{X}$ is any $n$-dimensional real Banach space and $0<\varepsilon < \frac{1}{3}$, then for every natural number 	
  		\begin{align*}
  			k\leq C\log n \frac{\varepsilon^2}{|\log \varepsilon|},
  		\end{align*}
  		there exists  a $k$-dimensional Banach subspace $\mathcal{Y}$ of $\mathcal{X}$ such that 
  		\begin{align*}
  			d_{BM}(\mathcal{Y}, (\mathbb{R}^k, \langle \cdot, \cdot \rangle))<1+\varepsilon.
  	\end{align*}}
  \end{theorem}
  Let $\mathcal{A}$  be a unital C*-algebra with invariant basis number property (see \cite{GIPSON} for a study on such C*-algebras)  and    $\mathcal{E}$, $\mathcal{F}$ be finite rank  Banach modules over $\mathcal{A}$  such that $\operatorname{rank} (\mathcal{E})=\operatorname{rank}(\mathcal{F})$. Modular Banach-Mazur distance between $\mathcal{E}$ and $\mathcal{F}$ is defined as 
  \begin{align*}
  	d_{MBM}(\mathcal{E}, \mathcal{F})\coloneqq \inf\{\|T\|\|T^{-1}\|: T: \mathcal{E} \to \mathcal{F} \text{ is invertible module homomorphism}\}.
  \end{align*}
  Given a unital C*-algebra $\mathcal{A}$ and $n \in \mathbb{N}$, by $\mathcal{A}^n$ we mean the standard (left) module over $\mathcal{A}$. We equip $\mathcal{A}^n$ with the C*-valued inner product $\langle \cdot, \cdot \rangle: \mathcal{A}^n\times \mathcal{A}^n \to \mathcal{A}$ defined by 
  \begin{align*}
  	\langle (a_j)_{j=1}^n, (b_j)_{j=1}^n \rangle \coloneqq \sum_{j=1}^na_jb_j^*, \quad \forall (a_j)_{j=1}^n, (b_j)_{j=1}^n \in 	\mathcal{A}^n.
  \end{align*}
  Hence norm on $\mathcal{A}^n$ is given by
  \begin{align*}
  	\|(a_j)_{j=1}^n\|\coloneqq \left\|\sum_{j=1}^na_ja_j^*\right\|^\frac{1}{2} , \quad \forall (a_j)_{j=1}^n \in \mathcal{A}^n.	
  \end{align*}
  Then it is well-known that $\mathcal{A}^n$ is a Hilbert C*-module. We denote this Hilbert C*-module by $(\mathcal{A}^n, \langle \cdot, \cdot \rangle)$. 
  
  \begin{problem}\label{MODULARDVORETZKYPROBLEM} \textbf{(Modular Dvoretzky Problem)
  		Let $\mathscr{A}$ be the set of all unital C*-algebras with invariant basis number property. What is the best function $\Psi: \mathscr{A} \times \left(0, \frac{1}{3}\right) \times \mathbb{N} \to (0, \infty)$  satisfying the following property: If $\mathcal{E}$ is any $n$-rank Banach module  over a unital C*-algebra $\mathcal{A}$ with IBN property and $0<\varepsilon < \frac{1}{3}$, then for every natural number 	
  		\begin{align*}
  			k\leq \Psi(\mathcal{A}, \varepsilon, n),
  		\end{align*}
  		there exists  a $k$-rank Banach submodule  $\mathcal{F}$ of $\mathcal{E}$ such that 
  		\begin{align*}
  			d_{MBM}(\mathcal{F}, (\mathcal{A}^k, \langle \cdot, \cdot \rangle))<1+\varepsilon.
  	\end{align*}}
  \end{problem}
  A particular case of Problem \ref{MODULARDVORETZKYPROBLEM} is the following conjecture. 
  \begin{conjecture}
  	\textbf{(Modular Dvoretzky Conjecture)
  		Let $\mathcal{A}$ be a unital C*-algebra with IBN property. 	There is a universal constant $C>0$ (which may depend upon $\mathcal{A}$) satisfying the following property: If $\mathcal{E}$ is any $n$-rank Banach module  and $0<\varepsilon < \frac{1}{3}$, then for every natural number 	
  		\begin{align*}
  			k\leq C\log n \frac{\varepsilon^2}{|\log \varepsilon|},
  		\end{align*}
  		there exists  a $k$-rank Banach submodule  $\mathcal{F}$ of $\mathcal{E}$ such that 
  		\begin{align*}
  			d_{MBM}(\mathcal{F}, (\mathcal{A}^k, \langle \cdot, \cdot \rangle))<1+\varepsilon.
  	\end{align*}}	
  \end{conjecture}
Our second kind of problems come from the type-cotype theory of Banach spaces \cite{LATALAOLESZKIEWICZ, SZAREK, HANDBOOK1, HANDBOOK2, ALBIACKALTON, LIQUEFFELEC1, LIQUEFFELEC2, DIESTELJARCHOWTONGE}.
  Let $\mathcal{H}$ be a Hilbert space, $n\in \mathbb{N}$. Recall that for any $n$ points $x_1, \dots , x_n \in \mathcal{H}$, we have 
  \begin{align}\label{H2}
  	\frac{1}{2^n}\sum_{\varepsilon_1, \dots,  \varepsilon_n \in \{-1, 1\}}	\left\|\sum_{j=1}^{n}\varepsilon_j x_j\right\|^2=\sum_{j=1}^{n}\|x_j\|^2.
  \end{align}
  It is Equality (\ref{H2}) which motivated the definition of type and cotype for Banach spaces. 
  \begin{definition}\cite{ALBIACKALTON}
  	Let $1\leq p \leq 2$. A Banach space $\mathcal{X}$ is said to be of \textbf{(Rademacher) type $p$} if there exists $T_p(\mathcal{X})>0$ such that 
  	\begin{align*}
  		\left(\frac{1}{2^n}\sum_{\varepsilon_1, \dots,  \varepsilon_n \in \{-1, 1\}}	\left\|\sum_{j=1}^{n}\varepsilon_j x_j\right\|^p\right)^\frac{1}{p}	\leq T_p(\mathcal{X})\left(\sum_{j=1}^{n}\|x_j\|^p\right)^\frac{1}{p},  \quad \forall 	x_1, \dots , x_n \in \mathcal{X}, ~\forall n \in \mathbb{N}.
  	\end{align*}
  \end{definition}
  \begin{definition}\cite{ALBIACKALTON}
  	Let $2\leq q  <\infty$. A Banach space $\mathcal{X}$ is said to be of  \textbf{(Rademacher) cotype $q$} if there exists $C_q(\mathcal{X})>0$ such that 
  	\begin{align*}
  		\left(\sum_{j=1}^{n}\|x_j\|^q\right)^\frac{1}{q}\leq C_q(\mathcal{X}) \left(\frac{1}{2^n}\sum_{\varepsilon_1, \dots,  \varepsilon_n \in \{-1, 1\}}	\left\|\sum_{j=1}^{n}\varepsilon_j x_j\right\|^q\right)^\frac{1}{q}, \quad \forall 	x_1, \dots , x_n \in \mathcal{X}, ~\forall n \in \mathbb{N}.
  	\end{align*}	
  \end{definition}
 Let $\mathcal{E}$ be a  (left) Hilbert C*-module over a unital C*-algebra $\mathcal{A}$, $n\in \mathbb{N}$. We see that  for any $n$ points $x_1, \dots , x_n \in \mathcal{E}$, we have 
  \begin{align}\label{HMODULE}
  	\frac{1}{2^n}\sum_{\varepsilon_1, \dots,  \varepsilon_n \in \{-1, 1\}}\left\langle \sum_{j=1}^{n}\varepsilon_j x_j, \sum_{k=1}^{n}\varepsilon_k x_k\right\rangle =	\sum_{j=1}^{n}\langle x_j, x_j \rangle. 
  \end{align}
  \begin{problem} 
  	\textbf{(Modular Type-Cotype  Problems) Whether there is a way to define type (we call modular-type) and cotype (we call modular-cotype) for Banach modules over C*-algebras which reduces to Equality (\ref{HMODULE}) for Hilbert C*-modules?}
  \end{problem}
 \begin{problem}
  	\textbf{Whether there is a notion of type and cotype for Banach modules over C*-algebras  such that Kwapien theorems holds?, In other words, whether following statements hold?
  		\begin{enumerate}[\upshape(i)]
  			\item A Banach module $\mathcal{M}$ over a unital C*-algebra $\mathcal{A}$ has modular-type 2 and modular-cotype 2 if and only if $\mathcal{M}$ is isomorphic to a Hilbert C*-module over $\mathcal{A}$.
  			\item If $\mathcal{M}$ and $\mathcal{N}$ are Banach modules over  a unital C*-algebra $\mathcal{A}$ of modular-type 2 and modular-cotype 2, respectively, then a  bounded module morphism $T:\mathcal{M} \to \mathcal{N}$ factors through a Hilbert C*-module over $\mathcal{A}$.
  	\end{enumerate}}
  \end{problem}
\begin{problem}
\textbf{(Modular Khinchin-Kahane Inequalities  Problems) Whether there is a Khinchin-Kahane	inequalities for Banach modules over C*-algebras  which reduce to Equality (\ref{HMODULE}) for Hilbert C*-modules?}
\end{problem}
  Our third kind of problems come from Grothendieck inequality \cite{BLEI, GROTHENDIECK, ALBIACKALTON, DIESTELFOURIESWART, FRIEDLANDLIMZHANG, RIETZ, HAAGERUP, JAMESON2, KAIJSER, BLEI2014, PISIERGROTHENDIECK}.
  \begin{theorem}\cite{BLEI, FRIEDLANDLIMZHANG, RIETZ, GROTHENDIECK, ALBIACKALTON,  DIESTELFOURIESWART} \textbf{(Grothendieck Inequality) There is a universal constant $K_G$ satisfying the following: For any Hilbert space $\mathcal{H}$ and any $m, n \in \mathbb{N}$, if a scalar matrix $[a_{j,k}]_{1\leq j \leq m, 1\leq k \leq n}$ satisfy 
  	\begin{align*}
  		\left|\sum_{j=1}^m \sum_{k=1}^na_{j,k}s_jt_k\right|\leq 1, \quad  \forall s_j , t_k \in \mathbb{K}, |s_j|\leq 1, |t_k|\leq 1, 
  	\end{align*}
  then 
  \begin{align*}
  \left| \sum_{j=1}^m \sum_{k=1}^na_{j,k}\langle u_j, v_k\rangle\right|	\leq K_G, \quad  \forall u_j , v_k \in \mathcal{H}, \|u_j\|\leq 1, \|v_k\|\leq 1. 
  \end{align*}}
  	 \end{theorem}
   \begin{problem}
 	\textbf{(Modular Grothendieck Inequality Problem - 1) 	Let $\mathscr{A}$ be the set of all unital C*-algebras. Let $\mathcal{E}$ be a Hilbert C*-module over a unital C*-algebra $\mathcal{A}$. Let $\mathcal{A}^+$ be the set of all positive elemnts in $\mathcal{A}$. What is the best function $\Psi: \mathscr{A} \times \mathbb{N} \times \mathbb{N} \to \mathcal{A}^+$  satisfying the following property: If $[a_{j,k}]_{1\leq j \leq m, 1\leq k \leq n}\in \mathbb{M}_{m\times n}(\mathcal{A})$ satisfy 
 		\begin{align*}
 		\left\langle \sum_{j=1}^m \sum_{k=1}^na_{j,k}s_jt_k, \sum_{p=1}^m \sum_{q=1}^na_{p,q}s_pt_q\right\rangle \leq &1, \quad  \forall s_j , t_k \in \mathcal{A},\\
 		& s_js_j^*=s_j^*s_j=1, \forall 1\leq j\leq m, t_kt_k^*=t_k^*t_k=1, \forall 1\leq k\leq n,
 	\end{align*}
 	then 
 	\begin{align*}
 		\left\langle \sum_{j=1}^m \sum_{k=1}^na_{j,k}\langle u_j, v_k\rangle, \sum_{p=1}^m \sum_{q=1}^na_{p,q}\langle u_p, v_q\rangle\right\rangle 	\leq &\Psi(\mathcal{A}, m, n), \quad  \forall u_j , v_k \in \mathcal{E},\\
 		& \langle u_j, u_j\rangle= 1, \forall 1\leq j\leq m, \langle v_k, v_k\rangle =1, \forall 1\leq k\leq n.
 	\end{align*}
 In particular, whether $\Psi$ depends on $m$ and $n$?}
 \end{problem}
 \begin{problem}
	\textbf{(Modular Grothendieck Inequality Problem - 2) 	Let $\mathscr{A}$ be the set of all unital C*-algebras. Let $\mathcal{E}$ be a Hilbert C*-module over a unital C*-algebra $\mathcal{A}$. Let $\mathcal{A}^+$ be the set of all positive elemnts in $\mathcal{A}$. What is the best function $\Psi: \mathscr{A} \times \mathbb{N} \times \mathbb{N} \to \mathcal{A}^+$  satisfying the following property: If $[a_{j,k}]_{1\leq j \leq m, 1\leq k \leq n}\in \mathbb{M}_{m\times n}(\mathcal{A})$ satisfy 
		\begin{align*}
			\left\langle \sum_{j=1}^m \sum_{k=1}^na_{j,k}s_jt_k, \sum_{p=1}^m \sum_{q=1}^na_{p,q}s_pt_q\right\rangle \leq 1, \quad  \forall s_j , t_k \in \mathcal{A}, \|s_j\|\leq 1, \forall 1\leq j\leq m, \|t_k\|\leq 1, \forall 1\leq k\leq n,
		\end{align*}
		then 
		\begin{align*}
			\left\langle \sum_{j=1}^m \sum_{k=1}^na_{j,k}\langle u_j, v_k\rangle, \sum_{p=1}^m \sum_{q=1}^na_{p,q}\langle u_p, v_q\rangle\right\rangle 	&\leq \Psi(\mathcal{A}, m, n), \quad  \forall u_j , v_k \in \mathcal{E},\\
			&\|u_j\|\leq  1, \forall 1\leq j\leq m, \|v_k\|\leq 1, \forall 1\leq k\leq n.
		\end{align*}
		In particular, whether $\Psi$ depends on $m$ and $n$?}
\end{problem}
 We believe strongly that $\Psi$ depends on $\mathcal{A}$.
  \begin{remark}
  	\textbf{Modular Bourgain-Tzafriri restricted invertibility conjecture} and \textbf{Modular Johnson-Lindenstrauss flattening conjecture}  have been stated in \cite{KRISHNA, KRISHNA2}.
  \end{remark}

 \bibliographystyle{plain}
 \bibliography{reference.bib}

\end{document}